\documentclass[reqno,11pt]{amsart}
\usepackage{latexsym}
\usepackage{amsfonts}
\usepackage{amssymb}
\usepackage{mathrsfs}
\usepackage[all]{xy}
\usepackage{bbm}

\newtheorem{theorem}{Theorem}[section]
\newtheorem{corollary}[theorem]{Corollary}
\newtheorem{lemma}[theorem]{Lemma}
\newtheorem{proposition}[theorem]{Proposition}
\numberwithin{equation}{section}

\newcommand{\RR}{\mathbb{R}}
\newcommand{\NN}{\mathbb{N}}
\newcommand{\QQ}{\mathbb{Q}}
\newcommand{\ZZ}{\mathbb{Z}}
\newcommand{\mt}{\mapsto}
\newcommand{\cI}{\mathcal{I}}
\newcommand{\cT}{\mathcal{T}}
\newcommand{\cS}{\mathcal{S}}
\newcommand{\cR}{\mathcal{R}}
\newcommand{\cF}{\mathcal{F}}
\newcommand{\cC}{\mathscr{P}}
\newcommand{\cV}{\mathcal{V}}
\newcommand{\cA}{\mathbf{A}}
\newcommand{\cB}{\mathbf{B}}
\newcommand{\cU}{\mathcal{U}}

\newcommand{\sF}{\mathscr{F}}

\newcommand{\Bad}{\mathbf{Bad}}
\newcommand{\br}{\mathbf{r}}
\newcommand{\suc}{\mathrm{suc}}

\begin{document}

\title[Badziahin-Pollington-Velani's theorem and Schmidt's game]
{Badziahin-Pollington-Velani's theorem \\ and Schmidt's game}

\author{Jinpeng An}
\address{LMAM, School of Mathematical Sciences \& Beijing International Center for Mathematical Research, Peking University, Beijing, 100871, China}
\email{anjinpeng@gmail.com}

\date{March 2012}
\thanks{Research supported by NSFC grant 10901005 and FANEDD grant 200915.}

\begin{abstract}
We prove that for any $s,t\ge0$ with $s+t=1$ and any $\theta\in\RR$ with $\inf_{q\in\NN}q^{\frac{1}{s}}\|q\theta\|>0$, the set of $y\in\RR$ for which $(\theta,y)$ is $(s,t)$-badly approximable is $\frac{1}{2}$-winning for Schmidt's game. As a consequence, we remove a technical assumption in a recent theorem of Badziahin-Pollington-Velani on simultaneous Diophantine approximation.
\end{abstract}

\maketitle

\section{Introduction}

Let $\br=(r_1,\ldots,r_d)$ be a $d$-tuple of nonnegative real numbers with $\sum_{i=1}^dr_i=1$, and let
\begin{equation}\label{E:def}
\Bad(\br)=\{(x_1,\ldots,x_d)\in\RR^d:\inf_{q\in\NN}\max_{1\le i\le d}q^{r_i}\|qx_i\|>0\}
\end{equation}
be the set of \emph{$\br$-badly approximable} vectors in $\RR^d$, where $\|\cdot\|$ denotes the distance of a real number to the nearest integer.
It is known that:
\begin{itemize}
  \item $\Bad(\br)$ has Lebesgue measure zero.
  \item $\Bad(\br)$ is a \emph{thick} subset of $\RR^d$, that is, it has full Hausdorff dimension at any point in $\RR^d$ (\cite{PV,Sc2}).
  \item When $\br=(\frac{1}{d},\ldots,\frac{1}{d})$, $\Bad(\br)$ is $\frac{1}{2}$-winning for Schmidt's game (\cite{Sc1,Sc2}).
  \item In the general case, $\Bad(\br)$ is a winning set for a modified Schmidt game (\cite{KW}).
\end{itemize}

When $d=2$, W. M. Schmidt conjectured in \cite{Sc3} that $\Bad(\br)\cap\Bad(\br')\ne\emptyset$ for any $\br$ and $\br'$. This conjecture was recently proved by D. Badziahin, A. Pollington and S. Velani \cite{BPV}. They actually proved the following much more general theorem. For a pair $\br=(s,t)$ of nonnegative numbers with $s+t=1$ and $\theta\in\RR$, denote $$\Bad(s,t;\theta)=\{y\in\RR:(\theta,y)\in\Bad(s,t)\}.$$

\begin{theorem}[Badziahin-Pollington-Velani \cite{BPV}]\label{T:BPV}
Let $(s_n,t_n)_{n=1}^\infty$ be a countable sequence of pairs of nonnegative numbers with $s_n+t_n=1$, and let $s=\sup_{n\in\NN}s_n$. Suppose that
\begin{equation}\label{E:tech}
\liminf_{n\to\infty}\min\{s_n,t_n\}>0.
\end{equation}
Then for any $\theta\in\RR$ with $\inf_{q\in\NN}q^{\frac{1}{s}}\|q\theta\|>0$, the set
$\bigcap_{n=1}^\infty\Bad(s_n,t_n;\theta)$
is thick in $\RR$.
\end{theorem}

Badziahin, Pollington and Velani wrote in \cite{BPV} that it would be desirable to remove assumption \eqref{E:tech} from the theorem. On the other hand, it is natural to expect that the theorem can be established using Schmidt's game. In this paper we prove that these are possible. Our main result is:

\begin{theorem}\label{T:main}
Let $s,t\ge0$ be such that $s+t=1$, and let $\theta\in\RR$ be such that
\begin{equation}\label{E:tech2}
\inf_{q\in\NN}q^{\frac{1}{s}}\|q\theta\|>0.
\end{equation}
Then $\Bad(s,t;\theta)$ is $\frac{1}{2}$-winning.
\end{theorem}
Some remarks about assumption \eqref{E:tech2} are in order.
\begin{itemize}
  \item When $s=0$, the left hand side of \eqref{E:tech2} is understood to be $\infty$.
  \item If $s>0$ and \eqref{E:tech2} does not hold, then $\Bad(s,t;\theta)=\emptyset$ (\cite{BPV}).
  \item If $\theta\in\Bad$,\footnote{As usual, $\Bad=\{x\in\RR:\inf_{q\in\NN}q\|qx\|>0\}$ denotes the set of badly approximable numbers.} then \eqref{E:tech2} automatically holds.
\end{itemize}

The definition of winning sets will be reviewed in Section \ref{S:3}. Here we recall that for any $\alpha\in(0,1)$, a countable intersection of $\alpha$-winning sets is also $\alpha$-winning, and an $\alpha$-winning subset of $\RR^d$ is thick (\cite{Sc1}).
In view of these facts, it follows from Theorem \ref{T:main} that Theorem \ref{T:BPV} remains true without assumption \eqref{E:tech}. In other words, we have

\begin{corollary}\label{C:1}
Let $(s_n,t_n)_{n=1}^\infty$ be a countable sequence of pairs of nonnegative numbers with $s_n+t_n=1$, and let $s=\sup_{n\in\NN}s_n$.
Then for any $\theta\in\RR$ with $\inf_{q\in\NN}q^{\frac{1}{s}}\|q\theta\|>0$, the set
$\bigcap_{n=1}^\infty\Bad(s_n,t_n;\theta)$
is thick in $\RR$.\footnote{After a preliminary version of this paper was circulated, the author learned from Barak Weiss that E. Nesharim \cite{Ne} has proved that under the conditions of Corollary \ref{C:1}, the set $\bigcap_{n=1}^\infty\Bad(s_n,t_n;\theta)$ has a nonempty intersection with the support of every measure on $\RR$ satisfying a power law, and in particular is uncountable.}
\end{corollary}

It was proved in \cite{BPV} that under assumption \eqref{E:tech}, the set $\bigcap_{n=1}^\infty\Bad(s_n,t_n)$ is thick in $\RR^2$. This was derived from Theorem \ref{T:BPV} using Marstrand's slicing theorem (\cite{Fa,KM}) and the thickness of $\Bad$. By the same reasons,
Corollary \ref{C:1} implies the following result.

\begin{corollary}
Let $(s_n,t_n)_{n=1}^\infty$ be a countable sequence of pairs of nonnegative numbers with $s_n+t_n=1$.
Then $\bigcap_{n=1}^\infty\Bad(s_n,t_n)$ is a thick subset of $\RR^2$.
\end{corollary}

The proof in \cite{BPV} of Theorem \ref{T:BPV} uses the dual form representation of $\Bad(\br)$ and is based on estimates of numbers of certain good intervals. Our proof of Theorem \ref{T:main} uses some ideas from \cite{BPV}, especially the idea of constructing good intervals. But we work with the simultaneous form \eqref{E:def} directly. This enables us to obtain better estimates for numbers of good intervals. With the help of the estimates, we prove that when playing Schmidt's game, the first player can choose good intervals and win the game.

The family of good intervals has the structure of a rooted tree (such a family is called ``tree-like" in \cite{KM,KW}). In this paper, we use the language of trees and represent intervals as vertices of a rooted tree. By doing this, we can employ K\"{o}nig's lemma in graph theory to give a simple proof of a structural property of the family of good intervals. In Section \ref{S:2}, we will give some preliminaries on rooted trees. Theorem \ref{T:main} will be proved in Sections \ref{S:3}--\ref{S:4}.

\section{Preliminaries on rooted trees}\label{S:2}

We first fix some notation and terminology. Recall that a \emph{rooted tree} is a connected graph $\cT$ without cycles and with a distinguished vertex $\tau_0$, called the \emph{root} of $\cT$. We identify $\cT$ with the set of its vertices. Any vertex $\tau\in\cT$ is connected to $\tau_0$ by a unique path. The length of the path is called the \emph{height} of $\tau$. The set of vertices of height $n$ is called the $n$th \emph{level} of $\cT$ and is denoted by $\cT_n$. Thus $\cT_0=\{\tau_0\}$. Let $\tau,\tau'\in\cT$. We write $\tau\prec\tau'$ to indicate that the path between $\tau_0$ and $\tau$ passes through $\tau'$. In this case, $\tau$ is called a \emph{descendant} of $\tau'$, and $\tau'$ is called an \emph{ancestor} of $\tau$. By definition, every vertex is a descendant and an ancestor of itself. For $\cV\subset\cT$, we write $\tau\prec\cV$ if $\cV$ contains an ancestor of $\tau$. If $\tau\prec\tau'$ and the height of $\tau$ is one greater than that of $\tau'$, then $\tau$ is called a \emph{successor} of $\tau'$, and $\tau'$ is called the \emph{predecessor} of $\tau$. Let $\cT(\tau)$ denote the rooted tree formed by all descendants of $\tau$. The root of $\cT(\tau)$ is $\tau$. Denote $\cT_{\suc}(\tau)=\cT(\tau)_1$, which is the set of all successors of $\tau$. More generally, for $\cV\subset\cT$, we denote $\cT_{\suc}(\cV)=\bigcup_{\tau\in\cV}\cT_{\suc}(\tau)$. In this paper, we use the convention that a subtree of $\cT$ has the same root as $\cT$. Thus $\cT(\tau)$ is not regarded as a subtree of $\cT$ unless $\tau=\tau_0$.

Let $N\in\NN$. We say that a rooted tree is \emph{$N$-regular} if every vertex has exactly $N$ successors. Note that an $N$-regular rooted tree is necessarily infinite. The following proposition will be needed later.

\begin{proposition}\label{P:tree}
Let $\cT$ be an $N$-regular rooted tree, $\cS\subset\cT$ be a subtree, and $1\le m\le N$ be an integer. Suppose that for every $m$-regular subtree $\cR$ of $\cT$, $\cS\cap\cR$ is infinite. Then $\cS$ has an $(N-m+1)$-regular subtree.
\end{proposition}

This proposition is motivated by \cite{BPV} and can be proved using the method in \cite[Section 7.3]{BPV}. Here we give a proof using K\"{o}nig's lemma (see \cite[Lemma 8.1.2]{Di}).\footnote{Nesharim \cite{Ne} has a proof which is essentially the same as ours.} In our context, K\"{o}nig's lemma states that if every level of an infinite rooted tree is finite, then the tree has an infinite path starting from the root. It is worth noting that the $m=N$ case of Proposition \ref{P:tree} reduces to a special case of K\"{o}nig's lemma.

We first prove a finite version of Proposition \ref{P:tree}. For $N,h\in\NN$, we say that a finite rooted tree $\cT$ is \emph{$(N,h)$-regular} if $\cT_{h+1}=\emptyset$ and every vertex of height less than $h$ has exactly $N$ successors.

\begin{lemma}\label{L:sub}
Let $\cT$ be an $(N,h)$-regular rooted tree, $\cS\subset\cT$ be a subtree, and $1\le m\le N$ be an integer. Suppose that for every $(m,h)$-regular subtree $\cR$ of $\cT$, $\cS_h\cap\cR_h\ne\emptyset$. Then $\cS$ has an $(N-m+1,h)$-regular subtree.
\end{lemma}

\begin{proof}
We proceed by induction on $h$. The $h=1$ case is obvious. Assume $h\ge2$ and the lemma holds if $h$ is replaced by $h-1$. Let
\begin{align*}
\cS'_1=\{\tau\in\cS_1: & \text{ for every $(m,h-1)$-regular subtree $\cR$ of $\cT(\tau)$}, \\
& \ \cS(\tau)_{h-1}\cap\cR_{h-1}\ne\emptyset\}.
\end{align*}
By the induction hypothesis, if $\tau\in\cS'_1$, then $\cS(\tau)$ has an $(N-m+1,h-1)$-regular subtree. It suffices to prove that
$\#\cS'_1\ge N-m+1$.
Suppose on the contrary that $\#\cS'_1\le N-m$. Let $\cV$ be a subset of $\cT_1\setminus\cS'_1$ with $\#\cV=m$. By the definition of $\cS'_1$, for every $\tau\in\cV$, we can choose an $(m,h-1)$-regular subtree $\cR_\tau$ of $\cT(\tau)$ such that $\cS(\tau)_{h-1}\cap(\cR_\tau)_{h-1}=\emptyset$ whenever $\tau\in\cS_1$. Let $\cR$ be the unique $(m,h)$-regular subtree of $\cT$ such that $\cR_1=\cV$ and $\cR(\tau)=\cR_\tau$ for every $\tau\in\cR_1$. Then
$$\cS_h\cap\cR_h=\bigcup_{\tau\in\cS_1\cap\cR_1}\cS(\tau)_{h-1}\cap\cR(\tau)_{h-1}=\bigcup_{\tau\in\cS_1\cap\cV}\cS(\tau)_{h-1}\cap(\cR_\tau)_{h-1}=\emptyset.$$
This contradicts the assumption of the lemma.
\end{proof}

Now we derive Proposition \ref{P:tree} from Lemma \ref{L:sub} and K\"{o}nig's lemma.

\begin{proof}[Proof of Proposition \ref{P:tree}]
Since every $(m,h)$-regular subtree of $\cT$ can be extended to an $m$-regular subtree, it follows from Lemma \ref{L:sub} that $\cS$ has an $(N-m+1,h)$-regular subtree for every $h\ge1$. We construct a rooted tree $\sF$ as follows. The root vertex of $\sF$ is the one point set $\{\tau_0\}$, where $\tau_0$ is the root of $\cT$. For $h\ge1$, $\sF_h$ is the nonempty finite set of $(N-m+1,h)$-regular subtrees of $\cS$, and $\cF\in\sF_h$ is a successor of $\cF'\in\sF_{h-1}$ if and only if $\cF'=\bigcup_{n=0}^{h-1}\cF_n$. Then $\sF$ is an infinite tree. By K\"{o}nig's lemma, $\sF$ has an infinite path starting from the root $\{\tau_0\}$. This means that there exists a family of subtrees $\{\cF(h)\in\sF_h\mid h\ge1\}$ of $\cS$ such that $\cF(h)=\bigcup_{n=0}^h\cF(h+1)_n$.
It follows that $\bigcup_{h=1}^{\infty}\cF(h)$ is an $(N-m+1)$-regular subtree of $\cS$.
\end{proof}

\section{The winning strategy}\label{S:3}

We first recall the definitions of Schmidt's game and winning sets introduced in \cite{Sc1}. The game is played by two players, say Alice and Bob.\footnote{Here we follow \cite{KW} for the names of the players.} Given a complete metric space $X$ and two numbers $\alpha,\beta\in(0,1)$. Bob starts the game by choosing a closed ball $\cB_0\subset X$. After $\cB_n$ is chosen, Alice chooses a closed ball $\cA_n\subset \cB_n$ of radius $\alpha$ times the radius of $\cB_n$, and Bob chooses a closed ball $\cB_{n+1}\subset \cA_n$ of radius $\beta$ times the radius of $\cA_n$. Then we have a nested sequence $\cB_0\supset\cA_0\supset\cB_1\supset\cA_1\supset\cdots$. A subset $Y\subset X$ is \emph{$(\alpha,\beta)$-winning} if Alice can play so that the single point in $\bigcap_{n=0}^\infty \cA_n$ lies in $Y$, and is \emph{$\alpha$-winning} if it is $(\alpha,\beta)$-winning for any $\beta\in(0,1)$. In the setting of Theorem \ref{T:main}, we have $X=\RR$. Thus $\cA_n$ and $\cB_n$ are compact intervals. In this section, we describe the winning strategy for Alice and derive Theorem \ref{T:main} from Proposition \ref{P:main} below.

Firstly, we have
\begin{equation}\label{E:01}
\Bad(0,1)=\RR\times\Bad, \quad \Bad(1,0)=\Bad\times\RR.\footnote{In some literatures, \eqref{E:01} is used as definition. As is well known, this can be also derived from \eqref{E:def}. For this, it suffices to show that if $\inf_{q\in\NN}q\|qx\|=0$, then $\inf_{q\in\NN}\max\{q\|qx\|,\|qy\|\}=0$ for every $y$. For any $\epsilon>0$, choose $q\in\NN$ such that $q\|qx\|<\epsilon^3$. By Dirichlet's theorem, there exists $n\in\NN$, $n\le\epsilon^{-1}$, such that $\|nqy\|<\epsilon$. Let $q'=nq$. Then $q'\|q'x\|\le n^2q\|qx\|<\epsilon^{-2}\epsilon^3=\epsilon$. Hence $\max\{q'\|q'x\|,\|q'y\|\}<\epsilon$.}
\end{equation}
So $\Bad(0,1;\theta)=\Bad$ and, when $\theta\in\Bad$, $\Bad(1,0;\theta)=\RR$. By \cite{Sc1}, $\Bad$ is $\frac{1}{2}$-winning.
Thus, in proving Theorem \ref{T:main}, we may assume that $s,t>0$.

Let $\theta\in\RR$ be such that \eqref{E:tech2} holds, and let $\beta\in(0,1)$. We want to prove that $\Bad(s,t;\theta)$ is $(\frac{1}{2},\beta)$-winning. In the first round of the game, Bob chooses a compact interval $\cB_0$. Let Alice choose the closed subinterval $\cA_0\subset\cB_0$ with $|\cA_0|=\frac{1}{2}|\cB_0|$ arbitrarily, where $|\cdot|$ denote the length of an interval. In what follows, we describe a strategy for the choice of $\cA_n$ ($n\ge1$) such that
$\bigcap_{n=0}^\infty \cA_n\subset\Bad(s,t;\theta)$.

Let $l=|\cA_0|$, $R=16\beta^{-4}$,
\begin{equation}\label{E:c}
c=\min\left\{\inf_{q\in\NN}q^{\frac{1}{s}}\|q\theta\|, \frac{1}{4}lR^{-1}, \frac{1}{8}R^{-2-\frac{3}{t^2}}\right\}.
\end{equation}
Then $R>16$, $c>0$. We assume that when a rational point in $\RR^2$ is expressed as $(\frac{p}{q},\frac{r}{q})$, then $q>0$ and the integers $p$, $q$, $r$ are coprime. Let    
\begin{equation}\label{E:cC}
\cC=\left\{\left(\frac{p}{q},\frac{r}{q}\right)\in\QQ^2: \left|\theta-\frac{p}{q}\right|<\frac{c}{q^{1+s}}\right\}.
\end{equation}
For $P=(\frac{p}{q},\frac{r}{q})\in\cC$, denote $$\Delta(P)=\left\{y\in\RR:\left|y-\frac{r}{q}\right|<\frac{c}{q^{1+t}}\right\}.$$
Note that if $y\notin\bigcup_{P\in\cC}\Delta(P)$, then $\max\{q^s|q\theta-p|,q^t|qy-r|\}\ge c$ for every $(\frac{p}{q},\frac{r}{q})\in\QQ^2$. Thus
\begin{equation}\label{Badc2}
\RR\setminus\bigcup_{P\in\cC}\Delta(P)\subset\Bad(s,t;\theta).
\end{equation}

Let $\cT$ be an $[R]$-regular rooted tree with root $\tau_0$, where $[\ \cdot \ ]$ denotes the integer part of a real number. We choose and fix
an injective map $\cI$ from $\cT$ to the set of closed subintervals of $\cA_0$ satisfying the following conditions:
\begin{itemize}
  \item For any $n\ge0$ and $\tau\in\cT_n$, $|\cI(\tau)|=lR^{-n}$. In particular, $\cI(\tau_0)=\cA_0$.
  \item For $\tau,\tau'\in\cT$, if $\tau\prec\tau'$, then $\cI(\tau)\subset\cI(\tau')$.
  \item For any $\tau'\in\cT$, the interiors of the intervals $\{\cI(\tau):\tau\in\cT_{\suc}(\tau')\}$ are mutually disjoint, and $\bigcup_{\tau\in\cT_{\suc}(\tau')}\cI(\tau)$ is connected.
\end{itemize}
Note that for $n\ge1$ and $\tau'\in\cT_{n-1}$, any closed subinterval of $\cI(\tau')$ of length $2lR^{-n}$ must contain an $\cI(\tau)$ for some $\tau\in\cT_{\suc}(\tau')$.  Suppose that $\cC$ is partitioned into a disjoint union
\begin{equation}\label{E:part}
\cC=\bigcup_{n=1}^\infty\cC_n.
\end{equation}
We inductively define a subtree $\cS$ of $\cT$ associated with the partition.
Let $\cS_0=\{\tau_0\}$. If $\cS_{n-1}$ $(n\ge1)$ is defined, we let
\begin{equation}\label{E:S_n}
\cS_n=\{\tau\in\cT_{\suc}(\cS_{n-1}):\cI(\tau)\cap\bigcup_{P\in\cC_n}\Delta(P)=\emptyset\}.
\end{equation}
Then $\cS=\bigcup_{n=0}^\infty\cS_n$ is a subtree of $\cT$. Note that
\begin{equation}\label{E:tau}
\cI(\tau)\subset\RR\setminus\bigcup_{P\in\cC_n}\Delta(P), \qquad \forall n\ge1, \tau\in\cS_n.
\end{equation}
The point is that for a suitable partition of $\cC$, the intervals $\{\cI(\tau):\tau\in\cS_n\}$ can serve as candidates for $\cA_{4n}$. This can be assured by the following proposition.

\begin{proposition}\label{P:main}
There exists a partition \eqref{E:part} such that the associated tree $\cS$ has an $([R]-5)$-regular subtree.
\end{proposition}

The proof of Proposition \ref{P:main} will be given in the next section. In the rest of this section, we assume Proposition \ref{P:main} and prove Theorem \ref{T:main}.

\begin{proof}[Proof of Theorem \ref{T:main} from Proposition \ref{P:main}]\footnote{This proof is motivated by discussions with Nikolay Moshchevitin.}
Let $\cC=\bigcup_{n=1}^\infty\cC_n$ be a partition such that $\cS$ has an $([R]-5)$-regular subtree, say $\cS'$. We inductively prove that for every $n\ge0$,
\begin{equation}\label{E:state}
\text{Alice can choose $\cA_{4n}=\cI(\tau_n)$ for some $\tau_n\in\cS'_n$.}
\end{equation}

Since $\cA_0=\cI(\tau_0)$, \eqref{E:state} holds if $n=0$. Assume $n\ge1$ and Alice has chosen $\cA_{4(n-1)}=\cI(\tau_{n-1})$, where $\tau_{n-1}\in\cS'_{n-1}$. We call the $5$ intervals $\{\cI(\tau):\tau\in\cT_{\suc}(\tau_{n-1})\setminus\cS'_{\suc}(\tau_{n-1})\}$ dangerous intervals.
We first prove that 
\begin{align}\label{E:substate}
&\text{For $j\in\{0,1,2,3\}$, Alice can play so that $\cA_{4(n-1)+j}$, and hence } \notag\\ 
&\text{$\cB_{4(n-1)+j+1}$, contains at most $[5\cdot2^{-j}]$ dangerous intervals.}
\end{align}
If $j=0$, there is nothing to prove. Assume $1\le j\le 3$ and \eqref{E:substate} holds if $j$ is replaced by $j-1$. Thus $\cB_{4(n-1)+j}$ contains at most $[5\cdot2^{-j+1}]$ dangerous intervals. Divide $\cB_{4(n-1)+j}$ into two closed subintervals of equal length. Then Alice can choose $\cA_{4(n-1)+j}$ to be one of the subintervals so that it contains at most
$\left[\frac{1}{2}[5\cdot2^{-j+1}]\right]\le[5\cdot2^{-j}]$ dangerous intervals. This proves \eqref{E:substate}.

By letting $j=3$ in \eqref{E:substate}, we see that Alice can play so that $\cB_{4n}$ contains no dangerous intervals. Since $\cB_{4n}$ has length $2lR^{-n}$, it contains an $\cI(\tau_n)$ for some $\tau_n\in\cT_{\suc}(\tau_{n-1})$. It follows that $\tau_n\in\cS'_n$. So Alice can choose $\cA_{4n}=\cI(\tau_n)$. This completes the proof of \eqref{E:state}.

In view of \eqref{E:state}, \eqref{E:tau} and \eqref{Badc2}, we have
\begin{align*}
\bigcap_{n=0}^\infty \cA_n&=\bigcap_{n=1}^\infty \cA_{4n}=\bigcap_{n=1}^\infty\cI(\tau_n)\subset\bigcap_{n=1}^\infty\RR\setminus\bigcup_{P\in\cC_n}\Delta(P)\\
&=\RR\setminus\bigcup_{P\in\cC}\Delta(P)\subset\Bad(s,t;\theta).
\end{align*}
This proves the theorem.
\end{proof}

\section{Proof of Proposition \ref{P:main}}\label{S:4}

In order to construct the partition \eqref{E:part} required in Proposition \ref{P:main}, we consider non-vertical rational lines in $\RR^2$ of the form
$$L(A,B,C)=\left\{(x,y)\in\RR^2: y=\frac{Ax+C}{B}\right\},$$
where $A,B,C\in\ZZ$ and $B>0$. We assume that when such a line is expressed as above, then $A,B,C$ are coprime. With this convention, $A,B,C$ are uniquely determined by $L(A,B,C)$. We need a lemma from \cite{BPV}. Here we reproduce the proof for completeness.

\begin{lemma}[\textrm{[1, Lemma 1]}]\label{L:aug}
Let $P=(\frac{p}{q},\frac{r}{q})\in\cC$. Then there exists a non-vertical rational line $L_P=L(A_P,B_P,C_P)$ passing through $P$ such that
\begin{equation}\label{E:v*}
|A_P|\le q^s, \quad B_P\le q^t.
\end{equation}
\end{lemma}

\begin{proof}
It suffices to prove that there exist $A,B,C\in\ZZ$ with $|A|\le q^s$ and $0<B\le q^t$ such that $Ap-Br+Cq=0$.
Since the number of pairs $(a,b)\in\ZZ^2$ with $0\le a\le q^s$ and $0\le b\le q^t$ is greater than $q$, there exist two such pairs $(a_1,b_1)\ne(a_2,b_2)$ with $b_1\ge b_2$ such that $$(a_1p-b_1r)-(a_2p-b_2r)=(a_1-a_2)p-(b_1-b_2)r$$ is divisible by $q$. Let $A=a_1-a_2$, $B=b_1-b_2$. Then $|A|\le q^s$, $0\le B\le q^t$, $(A,B)\ne(0,0)$, and there exists $C\in\ZZ$ such that $Ap-Br+Cq=0$. It remains to prove that $B\ne0$. Suppose on the contrary that $B=0$. Then $A\ne0$ and $\frac{C}{A}=-\frac{p}{q}$. In view of \eqref{E:c} and \eqref{E:cC}, we have $$c\le\inf_{n\in\ZZ}n^\frac{1}{s}\|n\theta\|\le|A|^\frac{1}{s}|A\theta+C|=|A|^{1+\frac{1}{s}}|\theta+\frac{C}{A}|\le q^{1+s}|\theta-\frac{p}{q}|<c.$$
This is a contradiction.
\end{proof}

For each $P\in\cC$, we choose and fix a line $L_P$ satisfying the conclusion of Lemma \ref{L:aug}. For $n\ge1$, let $$H_n=4cl^{-1}R^n,$$
\begin{equation}\label{E:C_n}
\cC_n=\left\{P=\left(\frac{p}{q},\frac{r}{q}\right)\in\cC: H_n\le qB_P< H_{n+1}\right\}.
\end{equation}
By \eqref{E:c}, we have $H_1=4cl^{-1}R\le1$. So $\cC=\bigcup_{n=1}^\infty\cC_n$. We prove that this partition satisfies the requirement of Proposition \ref{P:main}.

In view of \eqref{E:v*} and \eqref{E:C_n}, for $P=(\frac{p}{q},\frac{r}{q})\in\cC_n$ we have
\begin{equation}\label{E:qB}
q\ge H_n^\frac{1}{1+t}, \quad B_P<H_{n+1}^{\frac{t}{1+t}}.
\end{equation}
We further partition $\cC_n$ into a finite disjoint union. Let
$$\lambda=\frac{3}{t^2}, \quad \mu=\frac{1}{t(1+t)},$$
\begin{equation}\label{E:C_n0}
\cC_{n,1}=\{P\in\cC_n: H_{n+1}^{\frac{t}{1+t}}R^{-\lambda}\le B_P<H_{n+1}^{\frac{t}{1+t}}\}.
\end{equation}
For $2\le k\le n$, let
\begin{equation}\label{E:C_nl}
\cC_{n,k}=\{P\in\cC_n: H_{n+1}^{\frac{t}{1+t}}R^{-\lambda-(k-1)\mu}\le B_P< H_{n+1}^{\frac{t}{1+t}}R^{-\lambda-(k-2)\mu}\}.
\end{equation}
It is easy to check that
$$H_{n+1}^{\frac{t}{1+t}}R^{-\lambda-(n-1)\mu}=H_1^{\frac{t}{1+t}}R^{-\frac{sn}{t}-\frac{3+2t}{t^2(1+t)}}\le1.$$
So $\cC_n=\bigcup_{k=1}^n\cC_{n,k}$. The following lemma states that ``effective" points in $\cC_{n,k}$ lie on a single line.

\begin{lemma}\label{L:const}
For any $n\ge1$, $1\le k\le n$ and $\tau\in\cS_{n-k}$, the map $P\mt L_P$ is constant on $$\cC_{n,k}(\tau)=\{P\in\cC_{n,k}:\cI(\tau)\cap\Delta(P)\ne\emptyset\}.$$
\end{lemma}

\begin{proof}
Let $P_1,P_2\in\cC_{n,k}(\tau)$, $P_1\ne P_2$. We need to prove that $L_{P_1}=L_{P_2}$. We divide the proof into several steps.

\textbf{Step 1.} Suppose $P_i=(\frac{p_i}{q_i},\frac{r_i}{q_i})$, $L_{P_i}=L(A_i,B_i,C_i)$, $i=1,2$. Let $\phi_i=\frac{A_i\theta+C_i}{B_i}$. We first prove the following auxiliary inequalities:
\begin{equation}\label{E:aug1}
\left|\frac{p_1}{q_1}-\frac{p_2}{q_2}\right|<\frac{c}{q_1^{1+s}}+\frac{c}{q_2^{1+s}},
\end{equation}
\begin{equation}\label{E:aug0}
\frac{c}{q_i^{1+t}}\le\frac{c}{q_iB_i}<\frac{1}{4}lR^{-n+k},
\end{equation}
\begin{equation}\label{E:aug2}
\left|\frac{r_1}{q_1}-\frac{r_2}{q_2}\right|<\frac{3}{2}lR^{-n+k},
\end{equation}
\begin{equation}\label{E:aug3}
|\phi_1-\phi_2|<2lR^{-n+k}.
\end{equation}
The first two inequalities can be verified as follows:
$$\left|\frac{p_1}{q_1}-\frac{p_2}{q_2}\right|\le\left|\theta-\frac{p_1}{q_1}\right|+\left|\theta-\frac{p_2}{q_2}\right|
<\frac{c}{q_1^{1+s}}+\frac{c}{q_2^{1+s}},$$
$$\frac{c}{q_i^{1+t}}\le\frac{c}{q_iB_i}\le\frac{c}{H_n}=\frac{1}{4}lR^{-n}<\frac{1}{4}lR^{-n+k}.$$
To prove the third one, we choose $y_i\in\cI(\tau)\cap\Delta(P_i)$. Then, in view of \eqref{E:aug0},
\begin{align*}
\left|\frac{r_1}{q_1}-\frac{r_2}{q_2}\right|&\le\left|y_1-\frac{r_1}{q_1}\right|+\left|y_2-\frac{r_2}{q_2}\right|+|y_1-y_2|
<\frac{c}{q_1^{1+t}}+\frac{c}{q_2^{1+t}}+|\cI(\tau)|\\
&<\frac{1}{2}lR^{-n+k}+|\cI(\tau)|=\frac{3}{2}lR^{-n+k}.
\end{align*}
This proves \eqref{E:aug2}. By \eqref{E:aug0}, we have
$$\left|\phi_i-\frac{r_i}{q_i}\right|=\frac{|A_i|}{B_i}\left|\theta-\frac{p_i}{q_i}\right|<\frac{q_i^s}{B_i}\frac{c}{q_i^{1+s}}=\frac{c}{q_iB_i}
<\frac{1}{4}lR^{-n+k}.$$
It follows from this inequality and \eqref{E:aug2} that
$$|\phi_1-\phi_2|\le\left|\phi_1-\frac{r_1}{q_1}\right|+\left|\phi_2-\frac{r_2}{q_2}\right|+\left|\frac{r_1}{q_1}-\frac{r_2}{q_2}\right|
<2lR^{-n+k}.$$
Thus we obtain \eqref{E:aug3}.

\textbf{Step 2.} Assume $k=1$. In view of \eqref{E:C_n} and \eqref{E:C_n0}, it follows that
\begin{equation}\label{E:a1}
H_{n+1}^{\frac{1}{1+t}}R^{-1}<q_i< H_{n+1}^{\frac{1}{1+t}}R^\lambda.
\end{equation}
We prove that the line $L_{P_1}$ passes through $P_2$. Firstly, we have
\begin{align}\label{E:k11}
|A_1p_2-B_1r_2+C_1q_2|
=&q_2\left|A_1\left(\frac{p_2}{q_2}-\frac{p_1}{q_1}\right)-B_1\left(\frac{r_2}{q_2}-\frac{r_1}{q_1}\right)\right|\notag\\
\le & q_2|A_1|\left|\frac{p_1}{q_1}-\frac{p_2}{q_2}\right|+q_2B_1\left|\frac{r_1}{q_1}-\frac{r_2}{q_2}\right|.
\end{align}
By \eqref{E:aug1} and \eqref{E:a1},
\begin{equation}\label{E:k12}
q_2|A_1|\left|\frac{p_1}{q_1}-\frac{p_2}{q_2}\right|
<q_2q_1^s\left(\frac{c}{q_1^{1+s}}+\frac{c}{q_2^{1+s}}\right)
=c\left(\frac{q_2}{q_1}+\frac{q_1^s}{q_2^s}\right)
<2cR^{1+\lambda}.
\end{equation}
On the other hand, by \eqref{E:a1}, \eqref{E:qB} and \eqref{E:aug2},
\begin{equation}\label{E:k13}
q_2B_1\left|\frac{r_1}{q_1}-\frac{r_2}{q_2}\right|
<\frac{3}{2}H_{n+1}lR^{-n+1+\lambda}
=6cR^{2+\lambda}.
\end{equation}
Substituting \eqref{E:k12}, \eqref{E:k13} into \eqref{E:k11} and taking \eqref{E:c} into account, we obtain
$$|A_1p_2-B_1r_2+C_1q_2|<8cR^{2+\lambda}\le1.$$
This implies that $A_1p_2-B_1r_2+C_1q_2=0$. Thus $L_{P_1}$ is the line passing through $P_1$ and $P_2$. Similarly, $L_{P_2}$ is the line passing through $P_1$ and $P_2$. Hence $L_{P_1}=L_{P_2}$. This proves the lemma for the $k=1$ case.

\textbf{Step 3.} In what follows we assume $k\ge2$. In view of \eqref{E:C_n} and \eqref{E:C_nl}, we have
$$|A_i|\le q_i^s<H_{n+1}^{\frac{s}{1+t}}R^{s(\lambda+(k-1)\mu)}, \quad B_i< H_{n+1}^{\frac{t}{1+t}}R^{-(\lambda+(k-2)\mu)}.$$
To simplify the calculation, we denote $$M_A=H_{n+1}^{\frac{s}{1+t}}R^{s(\lambda+(k-1)\mu)}, \quad M_B=H_{n+1}^{\frac{t}{1+t}}R^{-(\lambda+(k-2)\mu)}.$$
Then $$|A_i|<M_A, \quad 1\le B_i<M_B.$$
We first verify
\begin{equation}\label{E:AB0}
M_B^2< M_A^sM_B^{2+s}< M_A^{1+t}M_B^{1+t}<\frac{1}{5}cl^{-1}R^{n-k}.
\end{equation}
In view of
$$\frac{M_B^s}{M_A^t}=\frac{R^{-s(\lambda+(k-2)\mu)}}{R^{st(\lambda+(k-1)\mu)}}<1,$$
we have
$$\frac{M_B^2}{M_A^sM_B^{2+s}}\le\frac{M_B^{1+\frac{1}{t}}}{M_A^sM_B^{2+s}}=\left(\frac{M_B^s}{M_A^t}\right)^\frac{s}{t}<1, \qquad \frac{M_A^sM_B^{2+s}}{M_A^{1+t}M_B^{1+t}}=\left(\frac{M_B^s}{M_A^t}\right)^2<1.$$
Thus the first two inequalities in \eqref{E:AB0} hold. It is easy to check that
\begin{align*}
5lM_A^{1+t}M_B^{1+t}=&5lH_{n+1}R^{(1-t^2)(\lambda+(k-1)\mu)-(1+t)(\lambda+(k-2)\mu)}\\
=&20cR^{-\frac{2}{t}-1}R^{n-k}< cR^{n-k}.
\end{align*}
Thus the third inequality in \eqref{E:AB0} holds as well.

\textbf{Step 4.} We prove $L_{P_1}=L_{P_2}$ by contradiction. Suppose $L_{P_1}\ne L_{P_2}$. We first consider the case where $L_{P_1}$ is parallel to $L_{P_2}$. In this case, we have $A_1B_2-A_2B_1=0$. Thus $B_1C_2-B_2C_1\ne0$. It follows that
$$|\phi_1-\phi_2|=\left|\frac{C_1}{B_1}-\frac{C_2}{B_2}\right|\ge\frac{1}{B_1B_2}>\frac{1}{M_B^2}.$$
Combining this inequality with \eqref{E:aug3}, we obtain
$$2lM_B^2R^{-n+k}>1.$$
But by \eqref{E:AB0} and \eqref{E:c},
$$2lM_B^2R^{-n+k}<c<1.$$
This is a contradiction.

\textbf{Step 5.} Now we suppose that $L_{P_1}$ is not parallel to $L_{P_2}$. Let $P_0=(\frac{p_0}{q_0},\frac{r_0}{q_0})\in\QQ^2$ be the intersection of $L(P_1)$ and $L(P_2)$. Then $A_1B_2-A_2B_1\ne0$ and
$$\frac{p_0}{q_0}=\frac{B_1C_2-B_2C_1}{A_1B_2-A_2B_1}, \quad \frac{r_0}{q_0}=\frac{A_1C_2-A_2C_1}{A_1B_2-A_2B_1}.$$
In particular, the nonzero integer $A_1B_2-A_2B_1$ is divisible by $q_0$. Thus
\begin{equation}\label{E:q0}
q_0\le|A_1B_2-A_2B_1|<2M_AM_B.
\end{equation}
We first verify that $P_0\in\cC$. In view of $$|\phi_1-\phi_2|=\left|\frac{A_1}{B_1}-\frac{A_2}{B_2}\right|\left|\theta-\frac{p_0}{q_0}\right|=\frac{|A_1B_2-A_2B_1|}{B_1B_2}\left|\theta-\frac{p_0}{q_0}\right|
\ge\frac{q_0}{M_B^2}\left|\theta-\frac{p_0}{q_0}\right|,$$
we have, by \eqref{E:q0}, \eqref{E:aug3} and \eqref{E:AB0}, that
$$q_0^{1+s}\left|\theta-\frac{p_0}{q_0}\right|\le q_0^sM_B^2|\phi_1-\phi_2|<4lM_A^sM_B^{2+s}R^{-n+k}<c.$$
Hence $P_0\in\cC$.

\textbf{Step 6.} We prove that $\Delta(P_1)\subset\Delta(P_0)$. Suppose $y\in\Delta(P_1)$. Then by \eqref{E:aug0},
$$\left|y-\frac{r_1}{q_1}\right|<\frac{c}{q_1^{1+t}}<\frac{1}{4}lR^{-n+k}.$$
In view of this inequality and \eqref{E:q0}, it follows that
\begin{align}
q_0^{1+t}\left|y-\frac{r_0}{q_0}\right|
&\le q_0^{1+t}\left|y-\frac{r_1}{q_1}\right|+q_0^{1+t}\left|\frac{r_1}{q_1}-\frac{r_0}{q_0}\right|\notag\\
&< lM_A^{1+t}M_B^{1+t}R^{-n+k}+2M_A^tM_B^t|A_1B_2-A_2B_1|\left|\frac{r_1}{q_1}-\frac{r_0}{q_0}\right|.\label{E:mid}
\end{align}
In order to estimate the second term of the right hand side, we note that
$$\frac{r_i}{q_i}-\frac{r_0}{q_0}=\frac{A_i}{B_i}\left(\frac{p_i}{q_i}-\frac{p_0}{q_0}\right), \quad i=1,2.$$
Eliminating $\frac{p_0}{q_0}$ from the two equations, we obtain
$$\left(\frac{A_1}{B_1}-\frac{A_2}{B_2}\right)\left(\frac{r_1}{q_1}-\frac{r_0}{q_0}\right)=\frac{A_1}{B_1}\left(\frac{r_1}{q_1}-\frac{r_2}{q_2}\right)
-\frac{A_1A_2}{B_1B_2}\left(\frac{p_1}{q_1}-\frac{p_2}{q_2}\right).$$
Thus by \eqref{E:aug1}, \eqref{E:aug0} and \eqref{E:aug2},
\begin{align*}
|A_1B_2-A_2B_1|\left|\frac{r_1}{q_1}-\frac{r_0}{q_0}\right|
\le&|A_1|B_2\left|\frac{r_1}{q_1}-\frac{r_2}{q_2}\right|+|A_1A_2|\left|\frac{p_1}{q_1}-\frac{p_2}{q_2}\right|\\
<&\frac{3}{2}lM_AM_BR^{-n+k}+|A_1A_2|\left(\frac{c}{q_1^{1+s}}+\frac{c}{q_2^{1+s}}\right)\\
=&\frac{3}{2}lM_AM_BR^{-n+k}+|A_2|B_1\frac{|A_1|}{q_1^s}\frac{c}{q_1B_1}+|A_1|B_2\frac{|A_2|}{q_2^s}\frac{c}{q_2B_2}\\
<&2lM_AM_BR^{-n+k}.
\end{align*}
Substituting this into \eqref{E:mid} and using \eqref{E:AB0}, we obtain
$$q_0^{1+t}\left|y-\frac{r_0}{q_0}\right|<5lM_A^{1+t}M_B^{1+t}R^{-n+k}< c.$$
Thus $y\in\Delta(P_0)$. This proves $\Delta(P_1)\subset\Delta(P_0)$.

\textbf{Step 7.} Let $n_0\ge1$ be the unique integer such that $P_0\in\cC_{n_0}$. We prove that
\begin{equation}\label{E:m}
n_0\ge n-k+1.
\end{equation}
Suppose on the contrary that $n_0\le n-k$. Let $\tau'$ be the unique vertex in $\cS_{n_0}$ such that $\tau\prec\tau'$. Then by \eqref{E:S_n}, $$\cI(\tau)\cap\Delta(P_1)\subset\cI(\tau')\cap\Delta(P_0)=\emptyset.$$ This contradicts $P_1\in\cC_{n,k}(\tau)$.

\textbf{Step 8.} In view of \eqref{E:qB} and \eqref{E:m}, we have $$q_0^{1+t}\ge H_{n_0}\ge H_{n-k+1}=4cl^{-1}R^{n-k+1}.$$
But by \eqref{E:q0} and \eqref{E:AB0},
$$q_0^{1+t}<4M_A^{1+t}M_B^{1+t}<cl^{-1}R^{n-k}.$$
This is a contradiction. Thus the proof of Lemma \ref{L:const} is completed.
\end{proof}

The purpose of Lemma \ref{L:const} is to establish the following estimate.

\begin{lemma}\label{L:<=2}
For any $n\ge1$, $1\le k\le n$ and $\tau'\in\cS_{n-k}$, the set
\begin{equation}\label{E:set1}
\{\tau\in\cT_n:\tau\prec\tau',\cI(\tau)\cap\bigcup_{P\in\cC_{n,k}}\Delta(P)\ne\emptyset\}
\end{equation}
contains at most two vertices.
\end{lemma}

\begin{proof}
Firstly, we claim that the set \eqref{E:set1} is a subset of
\begin{equation}\label{E:set2}
\{\tau\in\cT_n:\cI(\tau)\cap\bigcup_{P\in\cC_{n,k}(\tau')}\Delta(P)\ne\emptyset\}.
\end{equation}
In fact, if $\tau\in\cT_n$ satisfies $\tau\prec\tau'$ and $\cI(\tau)\cap\bigcup_{P\in\cC_{n,k}}\Delta(P)\ne\emptyset$, then there exists $P_0\in\cC_{n,k}$ such that
$$\cI(\tau')\cap\Delta(P_0)\supset\cI(\tau)\cap\Delta(P_0)\ne\emptyset.$$
It follows that $P_0\in\cC_{n,k}(\tau')$. Thus
$$\cI(\tau)\cap\bigcup_{P\in\cC_{n,k}(\tau')}\Delta(P)\supset\cI(\tau)\cap\Delta(P_0)\ne\emptyset.$$
This means that $\tau$ lies in the set \eqref{E:set2}.

Now we prove that \eqref{E:set2} contains at most two vertices. By Lemma \ref{L:const}, there exists $(A,B,C)\in\ZZ^3$ with $B>0$ such that for every $P=(\frac{p}{q},\frac{r}{q})\in\cC_{n,k}(\tau')$, $$|A|\le q^s, \quad B\le q^t, \quad Ap-Br+Cq=0, \quad qB\ge H_n.$$
For such $P$, if $y\in\Delta(P)$, then
\begin{align*}
\left|y-\frac{A\theta+C}{B}\right|&=\left|\left(y-\frac{r}{q}\right)-\frac{A}{B}\left(\theta-\frac{p}{q}\right)\right|
\le\left|y-\frac{r}{q}\right|+\frac{|A|}{B}\left|\theta-\frac{p}{q}\right|\\
& <\frac{c}{q^{1+t}}+\frac{q^s}{B}\frac{c}{q^{1+s}}\le\frac{2c}{qB}\le\frac{2c}{H_n}=\frac{1}{2}lR^{-n}.
\end{align*}
This implies that $\bigcup_{P\in\cC_{n,k}(\tau')}\Delta(P)$ is contained in the open interval
\begin{equation}\label{E:interval}
\left(\frac{A\theta+C}{B}-\frac{1}{2}lR^{-n},\frac{A\theta+C}{B}+\frac{1}{2}lR^{-n}\right),
\end{equation}
which has length $lR^{-n}$. Since the intervals $\{\cI(\tau):\tau\in\cT_n\}$ are of length $lR^{-n}$ and have mutually disjoint interiors, at most two of them intersect the interval \eqref{E:interval}. Thus the set \eqref{E:set2} contains at most two vertices.
\end{proof}

We now prove Proposition \ref{P:main} using Lemma \ref{L:<=2} and some ideas in the proof of \cite[Lemma 4]{BPV}.

\begin{proof}[Proof of Proposition \ref{P:main}]
By Proposition \ref{P:tree}, it suffices to prove that the intersection of $\cS$ with every $6$-regular subtree of $\cT$ is infinite. Let $\cR\subset\cT$ be a $6$-regular subtree, and denote $a_n=\#\cS_n\cap\cR_n$. Then $a_0=1$. We prove the infinity of $\cS\cap\cR$ by showing that for any $n\ge1$,
\begin{equation}\label{E:infinity}
a_n>2a_{n-1}.
\end{equation}

Let
$$\cU_n=\{\tau\in\cT_{\suc}(\cS_{n-1}\cap\cR_{n-1}):\cI(\tau)\cap\bigcup_{P\in\cC_n}\Delta(P)\ne\emptyset\}.$$
Then
\begin{align*}
\cS_n\cap\cR_n
=&\{\tau\in\cR_{\suc}(\cS_{n-1}\cap\cR_{n-1}):\cI(\tau)\cap\bigcup_{P\in\cC_n}\Delta(P)=\emptyset\}\\
=&\cR_{\suc}(\cS_{n-1}\cap\cR_{n-1})\setminus\cU_n.
\end{align*}
It follows that
\begin{equation}\label{E:infinity1}
a_n\ge6a_{n-1}-\#\cU_n.
\end{equation}
On the other hand,
\begin{align*}
\cU_n=&\bigcup_{k=1}^n\{\tau\in\cT_{\suc}(\cS_{n-1}\cap\cR_{n-1}):\cI(\tau)\cap\bigcup_{P\in\cC_{n,k}}\Delta(P)\ne\emptyset\}\\
\subset&\bigcup_{k=1}^n\{\tau\in\cT_n:\tau\prec\cS_{n-k}\cap\cR_{n-k},\cI(\tau)\cap\bigcup_{P\in\cC_{n,k}}\Delta(P)\ne\emptyset\}\\
=&\bigcup_{k=1}^n\bigcup_{\tau'\in\cS_{n-k}\cap\cR_{n-k}}\{\tau\in\cT_n:\tau\prec\tau',\cI(\tau)\cap\bigcup_{P\in\cC_{n,k}}\Delta(P)\ne\emptyset\}.
\end{align*}
By Lemma \ref{L:<=2}, we have
\begin{equation}\label{E:infinity2}
\#\cU_n\le\sum_{k=1}^n2a_{n-k}.
\end{equation}
From \eqref{E:infinity1} and \eqref{E:infinity2}, we obtain
\begin{equation}\label{E:infinity3}
a_n\ge6a_{n-1}-\sum_{k=1}^n2a_{n-k}.
\end{equation}
By letting $n=1$ in \eqref{E:infinity3}, we see that $a_1\ge4$. So \eqref{E:infinity} holds for $n=1$. Assume $n\ge2$ and \eqref{E:infinity} holds if $n$ is replaced by $1,\ldots,n-1$.
Then for any $1\le k\le n$, we have $$a_{n-k}\le 2^{-k+1}a_{n-1}.$$ Substituting this into \eqref{E:infinity3}, we obtain
$$a_n\ge6a_{n-1}-2a_{n-1}\sum_{k=1}^n2^{-k+1}>2a_{n-1}.$$
This completes the proof of \eqref{E:infinity}.
\end{proof}

\section*{Acknowledgments} The author is grateful to Dzmitry Badziahin, Dmitry Kleinbock, Andrew Pollington, Sanju Velani and Barak Weiss for valuable conversations and suggestions, especially to B. Weiss for informing the author of the work of Nesharim \cite{Ne}. A special thank goes to Nikolay Moshchevitin. After seeing an early version of the paper, which only gave $\frac{1}{8}$-winning property, he pointed out that Proposition \ref{P:main} implies $\frac{1}{2}$-winning property, and generously let the author incorporate his idea into this paper.

\end{document}